\documentclass[12pt]{amsart}
\usepackage{epsfig}
\usepackage{color}
\def\RR{\mathbb{R}}

\newcommand{\lp}{\left(}
\newcommand{\rp}{\right)}

\newtheorem{theorem}{Theorem}[section]
\newtheorem{lemma}[theorem]{Lemma}

\newtheorem{definition}{Definition}[section]

\newtheorem{remark}{Remark}[section]

\def\qed{\hbox{${\vcenter{\vbox{                        
   \hrule height 0.4pt\hbox{\vrule width 0.4pt height 6pt
   \kern5pt\vrule width 0.4pt}\hrule height 0.4pt}}}$}}

\newcommand{\be}{\begin{equation}}
\newcommand{\ee}{\end{equation}}
\newcommand{\bee}{\begin{equation*}}
\newcommand{\eee}{\end{equation*}}
\newcommand{\bea}{\begin{eqnarray}}
\newcommand{\eea}{\end{eqnarray}}
\newcommand{\bs}{\begin{split}}
\newcommand{\es}{\end{split}}

\numberwithin{equation}{section}
\begin{document}

\title[fractional geometric flow]{Weak and smooth solutions for a fractional Yamabe flow: the case of general compact and locally conformally flat manifolds}
\author[Daskalopoulos, Sire,  V\'azquez]{\rm Panagiota Daskalopoulos, \\Yannick Sire  \\ and Juan-Luis V\'azquez}
\date{}
\maketitle

\begin{abstract}

 As a counterpart of the classical Yamabe problem, a  fractional Yamabe flow has been introduced  by Jin and Xiong (2014) on the sphere.  Here we pursue its study  in the context of general compact smooth manifolds with positive fractional curvature.  First, we prove that the flow is locally well posed in the weak sense on any compact manifold. If the manifold is locally conformally flat  with positive Yamabe invariant, we also prove that the flow is smooth and converges to a constant scalar curvature metric. We provide different proofs using extension properties  introduced by Chang and Gonz\'alez (2011) for the conformally covariant fractional order operators.
\end{abstract}

\tableofcontents

\newpage

\section{Introduction}

Given  a compact Riemannian  manifold $(M, g_0)$   of dimension $n \geq 2$,  Hamilton introduced in \cite{hamilton} the following evolution for a metric $g(t)$
\begin{equation}
\left\{
\begin{array}{l}
\partial_t g(t) = -\Big ( R_{g(t)}-r_{g(t)}\Big )g(t)\\
g(0)=g_0,
\end{array}
\right .
\end{equation}
where $R_{g(t)}$ is the scalar curvature of $g(t)$ and
$$
r_{g(t)}=\text{vol}_{g(t)}(M)^{-1} \int_M R_{g(t)}\,d\text{vol}_{g(t)}.
$$
 This gave rise to an extensive literature, see e.g. \cite{chow,ye,struwe,brendle1,brendle2}.
On the other hand, in a seminal paper \cite{GZ} Graham and Zworski  constructed a family of conformally covariant operators $P^g_\gamma$, $\gamma \in (0,n/2)$,   on the conformal infinity of a Poincar\'e-Einstein manifold. These operators appear to be the higher-order generalizations of the conformal Laplacian. They coincide with the GJMS operators of \cite{GJMS}  for suitable integer values of $\gamma$. This paved the way to define an interpolated quantity $Q^g_\gamma$ for each $\gamma \in (0,n/2)$, which is just the   scalar curvature for $\gamma=1$, and the $Q$-curvature for $\gamma=2$ (see Section \ref{confop}). This new notion of curvature has been investigated in \cite{QG,CG, GMS}.

\subsection{The nonlocal flow}

 The Graham-Zworski approach motivates the introduction of a new flow problem with fractional curvature that replaces the curvature in Hamilton's Yamabe flow \cite{hamilton} by the new curvatures. The problem is posed as follows:
 {\sl Given  a compact Riemannian manifold $(M^n, g_0)$ of dimension $n \geq 2$ and given $\gamma \in (0,n/2)$,  to find an evolving metric $g(t)$ on $M $ such that
\begin{equation}\label{problem}
\left\{
\begin{array}{l}
\partial_t g = -k(t)\Big (Q^{g(t)}_\gamma-q^{g(t)}_\gamma \Big )g(t)\\
g(0)=g_0\,,
\end{array}
\right.
\end{equation}
where
$$
k(t)=\frac{n-2\gamma}{2n} \text{vol}_{g(t)}(M)^{\frac{2\gamma -n}{n}}
$$
and
$$q^g_\gamma= \text{vol}_{g(t)}(M)^{-1}\int_{M} Q^g_\gamma \,d\text{vol}_{g}\,.
$$
} \noindent Note that $k(t)$ and $q^g_\gamma$ depend only on $t$, and denote (see \cite{CG})
$$
Q_\gamma^g= P^g_\gamma(1).
$$
This flow is the gradient flow of the normalized total $\gamma$-curvature functional
\begin{equation}
 {\mathcal S}_\gamma(g) =\text{vol}_g(M)^{\frac{2\gamma -n}{n}}\int_M Q^g_\gamma\,d\text{vol}_g
\end{equation}
where $g \in [g_0], $  the conformal class of $g_0$, as observed in \cite{JX}. For $\gamma=1$, the  defined flow  is just the  Yamabe flow introduced by Hamilton.
 This new geometrical problem has been already considered  by Jin and Xiong in \cite{JX} where the authors investigate the flow on the sphere $M=\mathbb S^n$ with the round metric. They introduce the flow actually in this context but the generalization on any compact manifold $M$ is straightforward.
An important property of the previous flow is that it conserves the volume in time.

\subsection{General flow problem and results} The aim of the present paper is twofold. We first prove existence and uniqueness of mild and weak solutions of the fractional flow on any compact manifold with positive fractional curvature and then move on to the case of locally conformally flat manifolds with non-negative Yamabe invariant in the spirit of Ye's work \cite{ye}. For reasons which will become transparent later, we study the flow \eqref{problem} for  $\gamma \in (0,1)$, and this assumption is kept throughout the paper.

We prove the following results.

\begin{theorem}\label{LWP}
Assume that $g(0) \in [g_0]$ is smooth and that $M$ is an $n$-dimensional  smooth compact boundaryless manifold, being the conformal infinity of a Poincar\'e-Einstein manifold  ($X^{n+1},g_+)$. Assume also that $Q^{g_0}_\gamma \geq 0$ and $\lambda_1(g_+) \geq \frac{n^2}{2}-\gamma^2$. Then  the flow in \eqref{problem} with initial metric $g(0)$ exists for all times in the  sense of mild solutions and weak solutions if  $\gamma \in (0,1/2]$,  and provided $H=0$ in the case $\gamma \in (1/2,1)$.  Here $H$ denotes the mean curvature of $\partial_\infty X=M$.
\end{theorem}

 A number of remarks follow:

\noindent { (i) A reminder of the concepts of the Graham-Zworski theory is given in Section~\ref{sect2}.

\noindent (ii) In the previous theorem, by mild sense, we mean that the flow exists for all times using a semi-group approach. It basically means that, as soon as a contraction property is satisfied, the Crandall-Liggett theorem   \cite{CL} ensures the existence of a mild solution as limit of the Implicit Time Discretization Scheme. Then we connect to weak solutions of the flow.}

\noindent (iii) {\rm The present paper deals with flows with non-negative fractional curvarture. Due to our approach, we are not able to deal with negative $Q^{g_0}_\gamma$ curvature. This is due to the fact that in this case one cannot show contractivity of the semi-group approach in the Crandall-Liggett theory. We will leave it as an open problem and hope to investigate it in the next future.
However, a remark needs to be done in the negative case. Indeed, if one considers a constant $Q^{g_0}_\gamma$ curvature equal to $-1$ and space-independent solutions, one gets the following ODE for the un-rescaled flow on $M$ (see \eqref{extended2})
$$
\partial_t U^{N_\gamma} = -U, \quad N_\gamma= \frac{n+2\gamma}{n-2\gamma}.
$$
Since $N_\gamma >1$, one has two different solutions, one being trivial $U(t) \equiv 0$ and the other one non trivial $U(t)=k t^{1/(N_\gamma-1)}$. This is a counterexample to uniqueness.}


In order to state our next theorem, we introduce the well-known Yamabe constant: if $(M,g)$ is a compact manifold, then the Yamabe constant $Y(M)$ is the quantity
\begin{equation}
Y(M):=\inf \left \{ \frac{\int_M R_h dV_h}{\Big( \int_M dV_h\Big )^\frac{n-2}{n}},\,\,\,h \in [g]\right \}
\end{equation}
where $R_h$ is the scalar curvature of $M$ with respect to the metric $h$.

\begin{theorem}\label{GWP}
Let $\gamma \in (0,1)$ and assume that $g(0) \in [g_0]$ is smooth, has nonnegative fractional curvature, and that $M$ is locally conformally flat with positive Yamabe constant $Y(M)$. Then the flow in \eqref{problem} with initial metric $g(0)$ exists for all times $t \in (0,+\infty)$ and is smooth. Furthermore, there exists a smooth metric $g_\infty$ such that
\begin{equation}
\lim_{t \to +\infty} \|g(t)-g_\infty\|_{C^\ell}=0
\end{equation}
 for any integer $\ell$, and $Q^{g_\infty}_\gamma$ is constant.
\end{theorem}

\subsection{Formulation as a fractional fast  diffusion}
A main feature of the method developed in the present paper is to provide a solution to the  fractional  Yamabe problem by parabolic arguments. An important property of the operators $P^g_\gamma$ is their conformal covariance. More precisely,  as usual we write  $g=u^{4/(n-2\gamma)} \bar g$, then we have for any $f \in C^\infty(M)$
the conformal law (see \cite{GZ})
$$
P^{\bar g}_\gamma(uf)= u^{\frac{n+2\gamma}{n-2\gamma}}P^{g}_\gamma(f).
$$
Indeed, the metric $g_\infty$ in Theorem \ref{GWP} satisfies
$$
P_\gamma ^{g_\infty} u_\infty=Q_\gamma^{g_\infty}u_\infty^{\frac{n+2\gamma}{n-2\gamma}}\,\,\,\,\mbox{on }\,\,\, M,
$$
where $Q_\gamma^{g_\infty}$ is constant and $g_\infty=u_\infty^{\frac{4}{n-2\gamma}}g_0$. The fractional Yamabe problem has been investigated for $\gamma \in (0,1)$ in \cite{QG} and in the case of locally conformally flat manifolds in \cite{QR} with positive Yamabe constant (with an additional assumption on the Poincar\'e exponent of the Kleinian group). See also \cite{KMW}.
In the case of the sphere (or $\RR^n$), the classification of solutions of the Yamabe equation has been obtained in \cite{li,li2}.

As already noticed in the case of the standard Yamabe flow, the fractional Yamabe flow \eqref{problem} is related to a porous medium type equation. Indeed, one can eliminate the factor $k(t)$ by a time rescaling,  so that our nonlocal flow  changes into the following Cauchy problem
\begin{equation}\label{problemRescaled}
\left\{
\begin{array}{l}
\partial_t g = \Big (q^g_\gamma-Q^g_\gamma \Big )g\\
g(0)=g_0,
\end{array}
\right.
\end{equation}
where $t$ denotes now the new time.
Furthermore, if we write $g(t)=u^{\frac{4}{n-2\gamma}} g_0$, this Cauchy problem reduces to
\begin{equation}\label{problemu}
\left\{
\begin{array}{l}
\partial_t (u^{N_\gamma})  =-P_\gamma ^{g_0} u+q^{g(t)}_\gamma u^{N_\gamma}\,,\,\,  x \in M\\
u(0)=u_0,
\end{array}
\right.
\end{equation}
where $N_\gamma= \frac{n+2\gamma}{n-2\gamma}$,  up to a numerical constant that is absorbed into the time variable. Since $N_\gamma>1$, this is a {\sl fast diffusion equation } of fractional type on the manifold $M$, a very convenient formulation for  our calculations.

\medskip

$\bullet $ We recall that when $M=\RR^n$ \ with the flat metric (or when $M$ is the sphere through stereographic projection), Equation \eqref{problemu} becomes
\begin{equation}
\left\{
\begin{array}{l}
\partial_t u^{N_\gamma}  =-(-\Delta)^\gamma u\,,\quad \ x \in \RR^n,\\
u(0)=u_0,
\end{array}
\right.
\end{equation}
where $(-\Delta)^\gamma$ is the so-called fractional Laplacian, i.e., the Fourier multiplier with symbol $|\xi|^{2\gamma}$. Such an equation has been investigated in \cite{PMEADV} and \cite{PMECPAM}  and we will make use of several of their techniques in the present paper. Actually, in the paper \cite{PMECPAM}  the so-called non-rescaled flow  is thoroughly  investigated.

\medskip

\begin{remark}
{\rm It is important to notice that the case $\gamma=1/2$ corresponds actually to Escobar's problem \cite{escobar}. This has been emphasized in \cite{CG}.
This provides a new approach to the scalar curvature flow on manifolds with boundary. We refer the reader to  Brendle's \normalcolor \cite{brendleAsian} for the study of the Yamabe flow on manifolds with boundary.}
\end{remark}


\section{Conformal fractional Laplacians}\label{sect2}

\subsection{Poincar\'e-Einstein manifolds and Graham-Zworski theory}
Before proceeding further, we give a summary of the Graham-Zworski theory. Let $M$ be a compact manifold of dimension $n$ with a metric $\hat g$. Let $X^{n+1}$ be a compact manifold of dimension $n+1$ with boundary $M$. A function $\rho$ is a {\sl defining function } of $\partial X$ in $X$ if
$$\rho>0 \mbox{ in } X, \quad \rho=0 \mbox{ on }\partial X,\quad d\rho\neq 0 \mbox{ on } \partial X. $$
We say that $g^+$ is a conformally compact metric on $X$ with {\sl conformal infinity} $(M,[\hat g])$ if there exists a defining function $\rho$ such that the manifold $(\bar X,\bar g)$ is compact for the metric $\bar g=\rho^2 g^+$, and $\bar g|_M\in [\hat g]$, the conformal class of $\hat g$. If, in addition $(X^{n+1}, g^+)$ is a conformally compact manifold and $Ric[g^+] = -ng^+$, then we call $(X^{n+1}, g^+)$ a conformally compact Einstein manifold.  In the typical example $X$ is the Poincar\'e disk with the hyperbolic metric and $M$ is the infinite horizon at $|x|=1$.

It is well known that, given a conformally compact, asymptotically hyperbolic manifold $(X^{n+1}, g^+)$ and a representative $\hat g$ in $[\hat g]$ of the conformal infinity $M$, there is a defining function $\rho$ such that, on $M \times (0,\epsilon)$ in $X$, $g^+$ has the normal form \ $g^+ = \rho^{-2}(d\rho^2 + g_\rho)$ where $g_\rho$ is a one-parameter family of metrics on $M$ such that $g_\rho|_{M}=\hat g$. Moreover, $g_\rho$ has an asymptotic expansion which contains only even powers of $\rho$, at least up to degree $n$.

Graham-Zworski \cite{GZ} have shown that, given $f\in C^\infty(M)$ and  $s\in\mathbb C$, the eigenvalue problem
\be\label{equation-GZ}
-\Delta_{g^+} u-s(n-s)u=0,\quad\mbox{in } X
\ee
has a solution of the form
\be\label{general-solution}u = F \rho ^{n-s} + H\rho^s,\quad F|_{\rho=0}=f\ee
for all $s\in\mathbb C$ unless $s$ belongs to the spectrum of $\Delta_{g^+}$.
Moreover, it is known that
$$
\sigma(\Delta_{g^+}) = \left[(n/2)^2,\infty\right) \cup \sigma_{pp}(\Delta_{g^+})\,,
$$
where the pure point spectrum $\sigma_{pp}(\Delta_{g^+})$ (the set of $L^2$ eigenvalues) is finite and it is contained in
$\left(0,(n/2)^2\right)$.
Now, the {\sl scattering operator} on $M$ is defined as
$$
S(s)f = H|_M.
$$
 It is a meromorphic family of pseudo-differential operators in the half-plane $Re(s)>n/2$. The values $s = n/2, n/2 -1 , n/2 - 2, \ldots$ are simple poles of finite rank, these are known as the trivial poles; $S(s)$ has infinitely many other poles. However, for the rest of the paper we assume that we are not in those exceptional cases.

\subsection{Conformal fractional Laplacians}\label{confop}
Using the previous notations we define the conformally covariant fractional powers of the Laplacian as follows: for $s=\frac{n}{2}+\gamma$, $\gamma\in \lp 0,\frac{n}{2}\rp$, $\gamma\not\in \mathbb Z$, we set
\begin{equation}\label{P-operator}
P^{\hat g}_\gamma := d_\gamma S\lp\frac{n}{2}+\gamma\rp,\quad d_\gamma=2^{2\gamma}\frac{\Gamma(\gamma)}{\Gamma(-\gamma)}<0.
\end{equation}
The previous formula is a straightforward extension of \cite{GZ} (see also \cite{CG}).
In our framework, the idea is to see the compact (smooth connected) manifold $M$ as the boundary infinity of the asymptotical hyperbolic manifold $X^{n+1}$.  On the other, with this choice of multiplicative factor, the principal symbol of $P^{\hat g}_\gamma$ is exactly the principal symbol of the fractional Laplacian $(-\Delta_{\hat g})^{\gamma}$, precisely,
$$
\sigma( P^{\hat g}_\gamma)=|\xi|^{2\gamma}.
$$
We thus have that $P_\gamma^{\hat g}=  (-\Delta_{\hat g})^\gamma+\Psi_{\gamma-1}$, where we denote by $\Psi_m$ a pseudo-differential operator of order $m$. In the previous formula, the operator $(-\Delta_{\hat g})^{\gamma}$ is the fractional power of the Laplace-Beltrami operator $-\Delta_{\hat g}$ with respect to the metric $\hat g$.

When $\gamma$ is an integer, it turns out that the $P_{k}$ are the conformally invariant powers of the Laplacian constructed by Graham-Jenne-Mason-Sparling \cite{GJMS} and Fefferman-Graham \cite{FG}, that are local operators.
In particular, when $k=1$ we have the conformal Laplacian,
$$P_1=-\Delta_{\hat g} +\frac{n-2}{4(n-1)} R_{\hat g}$$
and when $k=2$, the Paneitz operator
$$
P_2=(-\Delta_{\hat g})^2 +\delta\lp a_n Rg+b_n Ric\rp d+\tfrac{n-4}{2}Q^n.
$$
The operators $P_\gamma^{\hat g}$ satisfy an important conformal covariance property (see \cite{GZ}). Indeed, for a conformal change of metric
\be\label{change-metric}\hat g_{w} =w^{\frac{4}{n-2\gamma}}\hat g,\ee
we have that
$$
P^{{\hat g}_w}_\gamma = w^{-\frac{n+2\gamma}{n-2\gamma}} P_\gamma^{\hat g} \lp u \,\cdot\rp.
$$

Finally,  the $Q_\gamma$-curvature of the metric associated to the functional $P_\gamma$ is defined by
$$
Q_\gamma^{\hat g}:=P_\gamma^{\hat g}(1).
$$
In particular, for a change of metric as \eqref{change-metric}, we obtain the equation for the $Q_\gamma$ curvature:
\begin{equation}
P^{\hat g}_{\gamma}(w)=w^{\frac{n+2\gamma}{n-2\gamma}} Q^{\hat g_w}_\gamma.
\end{equation}

\subsection{Connection to Dirichlet-to-Neumann operators} Let us now describe the Chang-Gonz\'alez extension property (see \cite{CG}). See also \cite{CafSil} in the flat case. This is what we really use in the present paper to investigate our flow.
The paper establishes a link between the just mentioned family of conformally covariant operators and Dirichlet-to-Neumann boundary operators corresponding to some uniformly degenerate elliptic operators. To be more precise: given an asymptotically hyperbolic manifold $(X^{n+1},g^+)$ and a representative $\hat h$ of the conformal infinity $(M^n,[\hat h])$, one can find a geodesic defining function $\rho$ such that the compactified metric can be written as
\begin{equation}    \label{g,h}
\bar g = \rho^2 g^+=d\rho^2+h_\rho .
\end{equation}

Consider now the following boundary value problem

\begin{equation}\label{div}\left\{\begin{split}
-\text{div}  (\rho^{1-2\gamma} \nabla U)  +E(\rho)U&=0\quad \mbox{in }(X^{n+1},\bar g), \\
U|_{\rho=0}&=f\quad \mbox{on }M^n,
\end{split}\right.
\end{equation}
where
$$E(\rho)=\rho^{-1-s}(\Delta_{g^+}-s(n-s))\rho^{n-s}\,,
$$
with $s=n/2+\gamma$. The following has been proved in \cite{CG}, Theorem 5.1.

\begin{lemma}\label{defining-function}
Let $(\bar X, \bar g)$ be a smooth $(n+1)$-dimensional compact manifold with boundary and let $\hat g$ be the restriction of the metric $\bar g$ to the boundary $M=\partial_\infty X. $ Let $\rho $ be a geodesic defining function. Then there exists an asymptotically hyperbolic metric $g^+$ with $\lambda_1(g_+) \geq \frac{n^2}{2}-\gamma^2$ on $X$ such that when $U$ solves \eqref{div}, then
\begin{enumerate}
\item For $\gamma \in (0,1/2)$,
$$P^{\hat g}_\gamma f=-d^*_\gamma \lim_{{\rho}\to 0} \rho^{1-2\gamma}\partial_{\rho} U$$
where the value of the constant $d^*_\gamma$ is given by
$$d^*_\gamma:=-\frac{2^{2\gamma-1}\Gamma(\gamma)}{\gamma \Gamma(-\gamma)}>0.$$
\item For $\gamma =1/2$
$$P^{\hat g}_\frac12 f=-\lim_{{\rho}\to 0}\partial_{\rho} U+\frac{n-1}{2}Hf,$$
where $H$ is the mean curvature of $M$.

\medskip

\item For $\gamma \in (1/2,1)$, if $H=0$ we  have
$$
P^{\hat g}_\gamma f=-d^*_\gamma \lim_{{\rho}\to 0} \rho^{1-2\gamma}\partial_{\rho} U.
$$
\end{enumerate}
\end{lemma}

A second useful lemma is the following (see \cite{CG}, Lemma 4.5 and Theorem 4.7)
\begin{lemma}\label{new-defining-function}
There exists a defining function $\rho^*$ such that $E(\rho^*)=0$ in the previous lemma. Then we have the following: let $U$ solve
\be\label{div2}\left\{\begin{split}
-div  ((\rho^*)^{1-2\gamma} \nabla U) &=0\quad \mbox{in }(X^{n+1}, g^*), \\
U|_{\rho=0}&=f\quad \mbox{on }M^n,
\end{split}\right.\ee
and derivatives are taken with respect to the metric $g^*=(\rho^*)^2 g^+$ where we still assume that $\lambda_1(g_+) \geq \frac{n^2}{2}-\gamma^2$. Then,
$$
P^{\hat g}_\gamma f=-d^*_\gamma \lim_{{\rho^*}\to 0} (\rho^*)^{1-2\gamma}\partial_{\rho^*} U+fQ^{\hat g}_\gamma\,.
$$
\end{lemma}

\begin{remark}
In the very interesting paper \cite{CaseChang}, the authors establish a link in the extension with conformally covariant operations in metric measured spaces. This generalizes the approach of \cite{CG}.
\end{remark}

\section{Short time existence for general compact manifolds}

In \cite{JX}, the short time existence  of the fractional flow on the sphere  follows from an implicit function argument, together with suitable estimates. The argument mainly relies on the fact that the problem is set in $\RR^n$ (after stereographic projection), and then they can use the representation of the fractional Laplacian as an integral with a singular kernel. In our general case of a compact manifold different from the sphere, such an approach does not seem easy to handle. We follow another route to prove the existence of weak solutions.

The proof of  our Theorem \ref{LWP}  is done in several steps. In order to prove local well-posedness it is enough to consider the un-rescaled flow, i.\,e., the following problem
\begin{equation}\label{problemulocalbis2}
\left\{
\begin{array}{l}
\partial_t  u^{N_\gamma}  =-P_\gamma ^{g_0}  u\,,\quad\mbox{for } \ t>0, \ x \in M,\\
u(0)= u_0.
\end{array}
\right.
\end{equation}

In Lemma \ref{defining-function}, the zero order term $E(\rho)$ is not suitable for our approach based on a contractivity argument. So instead, we use Lemma \ref{new-defining-function}. One then considers the extension to $X$ (the Poincar\'e-Einsetein manifold) of the function $u$ solving \eqref{problemulocalbis2} denoted $U$ such that then $\left.U\right|_{M}=u$. We drop the harmless constant $d^*_\gamma$ and  rewrite \eqref{problemulocalbis2} as:

Problem \eqref{problemulocalbis2} can be re-written as
\begin{equation}\label{extended2}
\left\{
\begin{array}{ll}
-\text{div}  ((\rho^*)^{1-2\gamma} \nabla U)=0\quad \mbox{in }(X^{n+1},g^*), \\
\lim_{{\rho^*}\to 0}(\rho^*)^{1-2\gamma}\partial_{\rho} U-UQ^{g_0}_\gamma=\partial_t (U^{N_\gamma}), \mbox{for }  x\in M, \ t>0\\
U(0,x)=u_0(x)\,, \quad x\in M. &
\end{array}
\right.
\end{equation}

\smallskip

\noindent {\sc Functional spaces.} We start by describing the functional spaces needed to define weak solutions. We introduce first the fractional Sobolev space suitable for our purposes. We define the semi-norm
$$
\|f\|_{\dot{H}^\gamma(M)}=\int_M f P^{g_0}_\gamma f
$$
Notice that this definition is consistent with the one on the flat case since $P^{|dx|^2}_\gamma=(-\Delta)^\gamma$. In view of the extension by Chang and Gonzalez, we will also need the following weighted Sobolev space
$$
H^1(M,\rho^{1-2\gamma})=\left \{ f \in L^2(M),\,\, | \,\, \rho^{1-2\gamma}|\nabla f|^2 \in L^1(M) \right \}.
$$
\smallskip

\noindent {\sc Notions of weak solutions.}
As already mentioned, we will use the formulation given in \eqref{extended2} to handle our problem. We then define a notion of weak solution for this boundary problem, which will serve as a weak formulation for problem  \eqref{problemulocalbis2}.

\begin{definition}\label{defWeak}
We say that $u$ in problem \eqref{problemulocalbis2} is a weak solution if its extension $U$ satisfies
\begin{enumerate}
\item $U \in C([0,\infty),L^1(M))$, $U^{1/N_\gamma} \in L^2_{loc}((0,\infty),\dot{H}^1(M,(\rho^*)^{1-2\gamma}))$
\item
the following identity holds for every test function $\varphi \in C^\infty_0(\bar X \times (0,\infty))$
$$
\int_0^\infty \int_M U^{N_\gamma} \partial_t \varphi + \int_0^\infty \int_M UQ^{g_0}_\gamma=\int_0^\infty \int_M \nabla_{g^*} U \cdot \nabla_{g^*}  \varphi
$$
\item the equality $U(0,.)=u_0$ holds a.e.
\end{enumerate}
\end{definition}

\noindent {\sc  Uniformly degenerate elliptic equations.} An important aspect of our theory relies on understanding the elliptic part of the flows. This is due to our approach based on Time Discretization. This can be done by studying the boundary elliptic problems in Lemma \ref{new-defining-function}. We collect here some results and definitions useful for us.
We are concerned with the uniformly degenerate elliptic equation
\begin{equation}\label{divTemp}
\left\{\begin{split}
-\text{div}  ((\rho^*)^{1-2\gamma} \nabla U) &=0\quad \mbox{in }(X^{n+1},\bar g^*), \\
U|_{\rho=0}&=f\quad \mbox{on }M^n.
\end{split}\right.
\end{equation}
We use the notations
\begin{align*}
& B_R^+=\{ (x,y)\in\mathbb R^{n+1} : y>0, |(x,y)|<R\}, \\
& \Gamma_R^0=\{ (x,0)\in\partial\mathbb R^{n+1}_+ : |x|<R\}, \text{ and}\\
& \Gamma_R^+=\{ (x,y)\in\mathbb R^{n+1} : y\ge 0, |(x,y)|=R\}.
\end{align*}
In local coordinates on $\Gamma_R^0$, the metric $\hat g$ writes $|dx|^2(1+O(|x|^2))$ where $x(p_0)=0$.
 Consider the matrix
$$A(x,y)=\sqrt{|det(\bar g^*)}y^a (\bar g^*)^{-1}$$
so that equation   \eqref{divTemp} writes in local coordinates
$$
\sum \partial_i( A_{ij} \partial_j U)=0.
$$
Furthermore, we have the crucial estimate
$$
A(x,y) \sim y^{a} \alpha(x,y) Id,
$$
where $\alpha(x,y)$ is uniformly elliptic.

\smallskip

The weight $y^a$ belongs to the Muckenhoupt class $A_2$ (see \cite{muck}), and the series of papers by Fabes, Kenig, Serapioni and Jerison \cite{FKS, FJK} provides a reasonably complete theory for  divergence elliptic equations with $A_2$ weights. In local coordinates in $\mathbb R^{n+1}$, Problem \eqref{divTemp} with its boundary condition writes
\begin{equation}\label{temp}
\left \{
\begin{array}{cc}
L_aU=\textrm{div\,} (A \nabla U)=0&\,\,\,\mbox{in $B_R^+$},\\
-y^a U_y=h&\,\,\,\mbox{on $\Gamma_R^0$.}
\end{array} \right .
\end{equation}

\begin{definition}
{\sl Given $R>0$ and a function $h \in L^1(\Gamma^0_R)$, we say that $u$ is a
  weak solution of (\ref{temp}) if
$$y^a |A\nabla u\cdot \nabla u|^2 \in L^1(B_R^+)$$
and
\begin{equation}
\int_{B_R^+} y^{a} A\nabla u \cdot \nabla \xi -\int_{\Gamma^0_R} h
\xi =0
\end{equation}
for all $\xi \in C^1(\overline{B_R^+})$ such that $\xi\equiv 0$ on $\Gamma^+_R$.}
\end{definition}

We have the following results.

\begin{theorem} [Solvability in Sobolev spaces \cite{FKS}] \label{solveFKS}
Let $\Omega\subset\RR^{n+1}$ be a smooth bounded domain,
$h=(h_1,...,h_{n+1})$ satisfy $|h|/|y|^a \in L^2(\Omega,|y|^a)$, and
$g \in H^1(\Omega,|y|^a)$. Then, there exists a unique solution
$u\in H^1(\Omega,|y|^a)$ of
$L_a u=-{\rm div }\, h$ in $\Omega$ with $u-g \in H^1_0(\Omega,|y|^a)$.
\end{theorem}

\begin{theorem} [H\"older local regularity \cite{FKS}]\label{HolderFKS}
Let $\Omega\subset\RR^{n+1}$ be a smooth bounded domain
and $u$ a solution of $L_a u=-{\rm div\,} h$ in
$\Omega$, where $|h|/|y|^a \in L^{2(n+1)}(\Omega,|y|^a)$.
Then, $u$ is H\"older continuous in $\Omega$ with a H\"older exponent depending only on $n$ and $a$.
\end{theorem}

\begin{theorem} [Harnack inequality \cite{FKS}] \label{HarnackFKS}
Let $u$ be a positive solution of $L_a u=0$ in
$B_{4R}(x_0)\subset\RR^{n+1}$. Then, $\sup_{B_R(x_0)} u \leq C \inf_{B_R(x_0)} u$
for some constant $C$ depending only on $n$ and $a$, and in particular, independent of~$R$.
\end{theorem}

\noindent {\sc Crandall-Liggett Scheme}. 
For our existence proof, we will use the Crandall-Liggett idea of implicit discretization in time to reduce the problem of existence of a so-called {\sl mild solution} for the evolution problem to a cascade of  elliptic problems. We will need an existence result plus suitable estimates.   We describe the principle of the Crandall-Ligget theorem in our context. Consider equation \eqref{problemulocalbis2},
$$
u_t=-P^{g_0}_\gamma (u^{m_\gamma})\,,
$$
where we set $u=v^{N_\gamma}$, and
$${m_\gamma}=\frac1{N_\gamma}=\frac{N-2\gamma}{N+2\gamma}<1.
$$
By taking a discrete sequence of times $0<t_1<t_2<\dots t_N$ and replacing the time derivative by an increment quotient, we reduce the previous evolution problem to a sequence of nonlinear elliptic problems of the iterative form
\begin{equation}
hP^{g_0}_\gamma u(t_k)^{m_\gamma}+ u(t_k)=u(t_{k-1}),\quad k=1,2, \cdots,
\end{equation}
We may take $t_k=kh$, and $h>0$  is the time step per iteration. We are led then to study the elliptic problems
\begin{equation}
hP^{g_0}_\gamma u^{m_\gamma}+ u=g\,,
\end{equation}
 posed in $M$.

\smallskip

\noindent {\sc Elliptic analysis.} Next, we perform the analysis of this elliptic problem. We consider only the case $Q^{g_0}_\gamma \geq 0$. There is no lack of generality in putting $h=1$ in the previous argument.

\begin{theorem}\label{Thelliptic}
Let $g \in L^1(M)\cap L^\infty(M)$. Then there exists a unique weak solution $U \in H^1_{loc}(X,(\rho^*)^{1-2\gamma})$ to
\begin{equation}\label{elliptic}
\left \{
\begin{array}{cl}
{\rm div}\, ((\rho^*)^{1-2\gamma} \nabla U)=0 \qquad \qquad &
{\mbox{ in $X$}}
\\
 -\lim_{{\rho^*}\to 0} (\rho^*)^{1-2\gamma}\partial_{\rho^*} U+Q^{g_0}_\gamma U+U^{N_\gamma} =g & {\mbox{ on $M$}}\\
\end{array}\right.
\end{equation}
Furthermore, if $U$ and $\tilde U$ are two weak solutions, one has the inequality
\begin{equation}
\int_M Q^{g_0}_\gamma (U -\tilde U )_+ +\int_M (U^{N_\gamma} -\tilde U^{N_\gamma} )_+ \leq  \int_M (g-\tilde g)_+.
\end{equation}
\end{theorem}
\begin{proof}
Pick a point $p \in \partial_\infty X$ and consider the problem in local coordinates around $p$ in a geodesic half-ball $B^+_R \subset X$
\begin{equation}\label{ellipticlocal}
\left \{
\begin{array}{c}
{\rm div}\, (A(x,y) \nabla U)=0 \qquad
{\mbox{ in $B_R^+$}}
\\
 - \lim_{y\to 0} y^{1-2\gamma}\partial_{y} U+Q_\gamma ^{g_0}U+U^{N_\gamma}  =g
\qquad{\mbox{ on $\Gamma^0_R$}}\\
\end{array}\right.
\end{equation}
We add homogeneous Dirichlet boundary conditions on $\Gamma^+_R$. A weak solution $U$ of \eqref{ellipticlocal} satisfies
$$\int_{B^+_R}  A(x,y)\nabla U \cdot \nabla \varphi + \int_{\Gamma^0_R} Q^{g_0}_\gamma U\varphi +\int_{\Gamma^0_R}  U^{N_\gamma}\varphi= \int_{\Gamma^0_R} g \varphi,$$
for each $\varphi$ compactly supported. A way to produce a weak solution is to minimize

$$J(U)=\frac12 \int_{B^+_R} A(x,y)\nabla U\cdot \nabla U+\frac12 \int_{\Gamma^0_R} Q^{g_0}_\gamma U^2+\frac{1}{N_\gamma +1}\int_{\Gamma^0_R}U^{N_\gamma +1}-\int_{\Gamma^0_R} Ug. $$

The functional is coercive by the Poincar\'e inequality (see \cite{FKS} for the Poincar\'e inequality for $A_2$ weights), the Sobolev trace embedding (see \cite{FKS}) and Cauchy-Schwarz inequality.  Recall here that the weight $A(x,y)$ is $A_2$ in $X$.

We now establish the contractivity.
Let $U$ and $\tilde U$ be two solutions with data $g$ and $\tilde g$.  Consider in the weak formulation the test function $\varphi=p(U-\tilde U)$ where $p$ is any smooth monotone approximation of the sign function, $0 \leq p(s)\leq 1$ and $p'(s) \geq 0$. Then we have by testing the weak formulation
$$\int_\Omega  p'(U-\tilde U)A(x,y)\nabla (U-\tilde U) \cdot \nabla (U-\tilde U)+ \int_M Q^{g_0}_\gamma  (U -\tilde U )p(U-\tilde U) $$
$$+\int_M (U^{N_\gamma} -\tilde U^{N_\gamma})p(U-\tilde U)  =\int_M (g-\tilde g)p(U-\tilde U).$$
Using the fact that
$$
A(x,y) \geq C y^{1-2\gamma}>0\,.
$$
and passing to the limit in the test function as $p$ tends to the Heaviside function $\widehat{p}=\mbox{sign}^+$, we get the desired inequality. This is what B\'enilan calls $T$-contractivity, \cite{benilan} and it implies both contractivity and the comparison principle.
\end{proof}

The proof of the existence theorem then follows from the fact mild solutions are weak solutions as in \cite{PMECPAM}.

\begin{remark}[Extinction in finite time of the un-rescaled flow]
{\rm The previous un-rescaled flow extinguishes in finite time. The proof is contained in  \cite{PMECPAM} in the Euclidean setting but carries out to the manifold case. The only point to check is the existence of a Stroock-Varopoulos inequality, i.e. let $\gamma \in (0,1)$ and $q>1$ then
$$
\int_M (|v|^{q-2}v)P^g_{\gamma/2} v \geq \frac{4(q-1)}{q^2}\int_{M} |P^g_{\gamma/4} |v^{q/2}|^2.
$$

The Stroock-Varopoulos inequality is a general inequality on Dirichlet spaces. Our operator $P_\gamma^g$ generates a Dirichlet form in $L^2(M)$ (see \cite{bakry}).}
\end{remark}

\subsubsection{Proof of Theorem \ref{LWP}}

The proof is as follows: given $u_0$, we compute a solution $u(x,\tau)$ of the un-rescaled flow \eqref{problemulocalbis2}. The rescaled flow satisfies

\begin{equation}\label{rescaledv}
\left\{
\begin{array}{l}
\partial_t v^{N_\gamma}  =-P_\gamma ^{g_0} v+a(t) v^{N_\gamma}\,\,\, t>0, x \in M\\
v(0)=v_0,
\end{array}
\right.
\end{equation}
and the flow is volume preserving. Now, if $F$ denotes
$$\frac{dF}{dt}=a(t)
$$
then it is easy to see that the solution $u(x,\tau)$ satisfies
\begin{equation}\label{rescaledv}
\left\{
\begin{array}{l}
\partial_\tau u^{N_\gamma}  =-P_\gamma ^{g_0} u\,\qquad t>0, x \in M\\
u(0)=u_0,
\end{array}
\right.
\end{equation}
provided
$$
\frac{d\tau}{dt}=e^{F(t)(1-\frac{1}{N_\gamma})}
$$
and furthermore we have $\int_M u(x,\tau)e^{F(t)}=$ constant, giving
$$
e^{F(\tau)}=\frac{\mbox{const}}{ \int_M u(t,\tau)}.
$$
Hence, from $u(x,\tau)$, we compute $\int_M u(t,\tau)$, which gives $F(\tau)$, giving $v(x,t)$ by the relation
$$
v=e^F u,
$$
hence the solution to the rescaled flow.

\section{Global existence and proof of Theorem \ref{GWP}}

Let $F$ be the stereographic projection from $\mathbb S^n$ into $\RR^n$ with north pole $q_0 \in \phi(V)$,i.e. the inverse of
$$
F^{-1}(x)=\Big ( \frac{2x}{1+|x|^2}, \frac{|x|^2-1}{|x|^2+1}\Big ),\,\,\,x \in \RR^n
$$
and we set $q_0=(0,\cdot \cdot\cdot, 1).$ Define $w$ as
$$
(F^{-1})^*\tilde g =w ^{4/n-2\gamma}g_{\RR^n}.
$$
Then the function $w$ satisfies the equation
\begin{equation}\label{problemulocalbis}
\left\{
\begin{array}{l}
\partial_t  w^{N_\gamma}  =-(-\Delta)^{\gamma} w\,\,\, t>0, x \in \RR^n\\
w(0)=w_0,
\end{array}
\right.
\end{equation}
Note that the function $w_0$ is only defined on $\RR^n \backslash F(\Gamma)$ and
$$\lim_{x \to F(\Gamma)}w_0(x)=+\infty. $$

This equation has been investigated in \cite{PMECPAM}. We can assume that the manifold $M$ is not conformally covered by the sphere. Indeed this case has been investigated in \cite{JX}.

By a deep theorem by Schoen and Yau \cite{schoen-yau}, since we assume that the locally conformally flat manifold has positive Yamabe constant,  there exists a conformal diffeomorphism $\phi$ from the universal cover $\tilde M$ of $(M,[g_0])$ onto a dense domain $\Omega$ of the sphere $\mathbb S^n$. Thus the manifold $(M,[g_0])$ is the quotient of $\Omega$ under a Kleinian group and $\Gamma=\partial \Omega$ is the limit set of this group (note that $\Gamma \neq \emptyset$).  We set
$$
\tilde g= (\phi^{-1})^*\pi^*g
$$
and
$$
\tilde g=\tilde u ^{\frac{4}{n-2\gamma}} g_{\mathbb S^n}
$$
where $\pi:\tilde M \to M$ is the covering map and $g_{\mathbb S^n}$ is the round metric on the sphere. Therefore, by construction, the metric $\tilde g$ solves \eqref{problem} and by conformality, $\tilde u $ solves \eqref{problemu} with $g_0$ replaced by $g_{\mathbb S^n}$.

The following lemma is a corollary of Proposition 2.6 in \cite{schoen-yau}.
\begin{lemma}\label{SY}
Let $M$ be locally conformally flat manifold with non-negative Yamabe invariant. Assume equation \eqref{problemu} has a local solution defined on $(0,T)$ for some $T>0$. Then for any $\tilde t \in [0,T)$, we have, uniformly in $t \in [0,\tilde t]$:
$$
\lim_{x \to \Gamma} \tilde u(x,t)=+\infty.
$$
\end{lemma}

Given $p_0\in M$, we choose a point $\tilde p_0 \in \tilde M$ and a neighborhood $V$ of $\tilde p_0$ such that $\pi(\tilde p_0)=p_0$ and $dist(\phi(V),\Gamma) >0$. Then there is some $C>0$ such that
$$
\tilde u_0=\tilde u(.,0) >C^{-1}
$$
and $\tilde u_0$ is smooth on $\phi(V)$.

The final step to prove Theorem \ref{GWP} is the following theorem.
\begin{theorem}\label{harnack}
Let $M$ be locally conformally flat manifold with non-negative Yamabe invariant. Let $u$ be a solution \eqref{problemu}. Then for any $t \geq 0$,  one has the Harnack inequality
$$
\sup_{x \in M}u \leq C \inf_{x \in M} u.
$$
\end{theorem}
The proof of Theorem \ref{harnack} follows the approach by Ye \cite{ye}. As a consequence of the volume-preserving character of the flow, there exists $\alpha,\beta>0$ such that for any $t \geq 0$ one has for any $x \in M$
$$\alpha < u(t,x) <\beta .$$

Hence the solution exists globally and the theorem is proved, is smooth by the results in \cite{vazquez} and the convergence of the flow is ensured by the well-known results of Simon \cite{simon}. We now come to the proof of Theorem \ref{harnack}.

This is based on the following theorem
\begin{theorem}
Let $\gamma \in (0,1)$ and assume that $M$ is locally conformally flat with nonnegative fractional curvature. Then there exists $C$ not depending on $u$ such that
\begin{equation}
\sup_{M}\frac{|\nabla_{g_0} u|}{u} \leq C.
\end{equation}
\end{theorem}
\begin{proof}
The proof follows the argument of \cite{ye}. The solution $\tilde u(t,x)$ has the expansion at infinity
$$
\tilde u(t,x)=\frac{2^{(n-2\gamma)/2}}{|x|^{n-2\gamma}}\Big ( a_0+\frac{a_ix_i}{|x|^2}+\Big ( a_{ij}-\frac{n-2\gamma}{2}\delta_{ij}\Big )\frac{x_i x_j}{|x|^4}+O(|x|^{-3})\Big ).
$$
A similar expansion holds for the derivatives of $\tilde u$:
$$
\frac{\partial \tilde u}{\partial x_i}(t,x)=2^{(n-2\gamma)/2}\Big (-\frac{n-2\gamma}{|x|^{n-2\gamma+2}}x_i\Big ( a_0+\frac{a_j x_j}{|x|^2}\Big )+ \frac{a_{i}}{|x|^{n-2\gamma+2}}-\frac{2x_i a_j x_j}{|x|^{n-2\gamma+4}}$$
$$
+O(|x|^{-(n-2\gamma+3)})\Big )
$$
where
$$
a_0(s)=u(q_0,t),
$$
$$
a_i(s)=\frac{\partial (u\circ G)}{\partial x_i}(0)
$$
and
$$
a_{ij}(s)=\frac{\partial^2 (u\circ G)}{2\partial x_i x_j}(0)
$$
where $q_0$ denotes the north pole and $G$ the inverse of the stereographic projection.

Denoting $y_i(t)$,
$$
y_i(t)=\frac{a_i}{(n-2\gamma)a_0}
$$
and $y(s)=(y_1(s),...,y_n(s))$, it is  enough to prove a uniform bound on $y$ in the local existence range. Fix a time $T$. After a rotation and a reflection, we may assume
$$
y_n(T)=\max_i |y_i(T)|.
$$
A standard argument gives that for some $\lambda_0 >0$ we have: for each $\lambda > \lambda_0$ the following holds
$$
\tilde u(x,0) >\tilde u (x^\lambda,0)\,\,\,x_n <\lambda
$$
where $x^\lambda=(x_1,....,x_{n-1},2\lambda-x_n)$, the reflection point w.r.t the hyperplane $\left \{ x_n=\lambda \right \}.$
Note that here we have a singular set $F(\Gamma)$ for $\tilde u(x,0)$. However, thanks to Lemma 4.1, the previous estimate holds.

We may assume that $F(\Gamma)$ lies strictly below the plane $\left \{ x_n=\lambda_0 \right \}.$ We now follow the proof of \cite{JX}.
By proposition 2.3 in \cite{JX}, one has uniformly in $[0,T)$
$$
\tilde u(x,t) >\tilde u (x^\lambda,t)\,\,\,x_n <\lambda, \lambda \geq \lambda_0 .
$$
We  claim that
$$
\max_{s \in [0,T]} y_n(s) < \lambda_0
$$
If not, there exists $\tilde s \in (0,T]$ such that $y_n(\tilde s)=\max_{0 \leq s \leq T} y_n(s) \leq \lambda_0$. Setting $\lambda=y_n(\tilde s)$ one gets after defining $\bar u(x,s)= \tilde u(x+y_n(\tilde s),s)$
$$
\bar u(x',x_n,s) >\bar u (x',-x_n,s)\,\,\,s\in [0,T],x_n <0.
$$
Then taking Kelvin transforms $\bar u_1(x,s)=\frac{1}{|x|^{n-2\gamma}}\bar u(x/|x|^2,s)$ (satisfying the same equation) we have
$$
\bar  u_1(x',x_n,s) >\bar u_1 (x',-x_n,s)\,\,\,s\in [0,T],x_n <0.
$$
Invoking proposition 2.8 in \cite{JX}, one has
$$
\frac{\partial (\bar u_1 (x',x_n,s)-\bar u_1(x',-x_n,s)}{\partial x_n}<0\,\,\,(x=0,s=\tilde s).
$$
This contradicts the asymptotic expansion of $\tilde u(x+y,s)$. Hence to derive the Harnack inequality, the previous argument gives
$$
\frac{|\nabla _{\mathbb S^n} \tilde u|}{ \tilde u} \leq C
$$
on $\Phi(V')$ where $V'$ is a neighborhood of $p_0$ with $V' \subset \subset V$. Now, by Lemma \ref{SY}, this leads
$$
\frac{|\nabla _{g_0} \tilde u|}{\ \tilde u} \leq C
$$
on $\pi(V')$. Since $M$ is compact, one can cover $M$ by finitely many $V'$ and up to enlarging $C$ this gives the theorem.
\end{proof}

\section*{Acknowledgments}
P. Daskalopoulos has  been partially supported by NSF grant DMS-1600658 and J.L. V\'azquez  by Project MTM2014-52240-P (Spain). The second author would like to acknowledge the hospitality of the department of Mathematics of Columbia University where part of this work was discussed.



\bibliographystyle{alpha}
\bibliography{biblioflow}

\

\medskip

\small

\begin{center}
P. Daskalopoulos, Department of Mathematics \\
Columbia University \\
2990 Broadway \\
New York, NY 10027\\
pdaskalo@math.columbia.edu
\end{center}

\begin{center}
Y. Sire, Department of Mathematics \\
Johns Hopkins University \\
3400 N. Charles Street \\
Baltimore, MD 21218\\
sire@math.jhu.edu
\end{center}

\begin{center}
J. L. V\'azquez, Departamento de Matem\'aticas,
\\
Universidad Aut\'onoma de Madrid,\\
Campus de Cantoblanco, \\
28049 Madrid, Spain \\
juanluis.vazquez@uam.es
\end{center}

\end{document}